\documentclass[11pt,a4paper]{article}
\usepackage[english]{babel}
\usepackage{amsmath,amsfonts,amsthm,amssymb,amsbsy,upref,color,graphicx,amscd,hyperref,makeidx,enumerate,bbm,comment}
\usepackage[a4paper,margin=2.80truecm]{geometry}
\usepackage[active]{srcltx}
\usepackage[latin1]{inputenc}

\usepackage{graphicx}
\usepackage{subfigure}

\def\ds{\displaystyle}
\def\eps{{\varepsilon}}

\def\O{\Omega}

\def\R{\mathbb{R}}

\def\E{\mathcal{E}}
\def\F{\mathcal{F}}

\def\LL{\mathcal{L}}
\def\M{\mathcal{M}}

\def\Dr{D}

\newcommand{\be}{\begin{equation}}
\newcommand{\ee}{\end{equation}}
\newcommand{\bib}[4]{\bibitem{#1}{\sc#2: }{\it#3. }{#4.}}
\newcommand{\cp}{\mathop{\rm cap}\nolimits}

\numberwithin{equation}{section}
\theoremstyle{plain}
\newtheorem{teo}{Theorem}[section]
\newtheorem{lemma}[teo]{Lemma}

\newtheorem{prop}[teo]{Proposition}

\theoremstyle{remark}
\newtheorem{oss}[teo]{Remark}

\newenvironment{ack}{{\bf Acknowledgements. }}

\title{Worst-case shape optimization for the Dirichlet energy}

\author{Jos\'e Carlos Bellido, Giuseppe Buttazzo, Bozhidar Velichkov}

\begin{document}

\maketitle

\begin{abstract}
We consider the optimization problem for a shape cost functional $F(\O,f)$ which depends on a domain $\O$ varying in a suitable admissible class and on a ``right-hand side'' $f$. More precisely, the cost functional $F$ is given by an integral which involves the solution $u$ of an elliptic PDE in $\O$ with right-hand side $f$; the boundary conditions considered are of the Dirichlet type. When the function $f$ is only known up to some degree of uncertainty, our goal is to obtain the existence of an optimal shape in the worst possible situation. Some numerical simulations are provided, showing the difference in the optimal shape between the case when $f$ is perfectly known and the case when only the worst situation is optimized.
\end{abstract}



\section{Introduction}\label{sintro}

In worst-case optimization problems one has two sets $X,Y$ of admissible choices and a cost functional $F:X\times Y\to\overline\R$; the goal is to minimize $F$ over $X$ when the worst choice with respect to $Y$ occurs. In other words, we consider the optimization problem
$$\min\big\{\F(x)\ :\ x\in X\big\}$$
where the cost functional $\F$ is defined by
$$\F(x)=\sup\big\{F(x,y)\ :\ y\in Y\big\}.$$
For a clear and extended presentation of worst-case optimization problems in structural mechanics we refer to \cite{all14}.

In the present paper we consider a worst-case shape optimization problem for elliptic PDEs with Dirichlet boundary conditions. More precisely, we fix a bounded domain $D\subset\R^d$ and for every domain $\O\subset D$ and $f\in L^2(D)$ we consider the state function $u$, solution of the PDE
$$-\Delta u=f\quad\hbox{in}\quad\O,\qquad u\in H^1_0(\O),$$
and a cost functional of the form 
$$F(\O,f)=\int_\O j(x,u)\,dx,$$
where $j:\Dr\times\R\to\R$ is a decreasing function in the second variable. If we assume that the right-hand side $f$ may vary under an unknown small perturbation, we obtain the worst-case shape functional
$$\F(\O)=\sup\big\{F(\O,f+g)\ :\ \|g\|_{L^2(D)}\le\delta\big\},$$
where $\delta>0$ is a fixed real number. In Theorem \ref{excontr} we show that for small $\delta$, there exists a solution to the shape optimization problem
$$\min\Big\{\F(\O)\ :\ \O\subset\Dr,\ |\O|\le m\Big\},$$
where we indicate by $|\cdot|$ the Lebesgue measure in $\R^d$.

Of particular interest is the case when $\ds j(x,u)=-\frac12 f(x) u$, when the cost functional $F(\O,f)$ becomes the Dirichlet energy
$$E(\O,f)=\min\left\{\int_\O\Big(\frac12|\nabla u|^2-fu\Big)\,dx\ :\ u\in H^1_0(\O)\right\},$$
and we denote by $\E$ the corresponding worst-case functional. We discuss this case in Section \ref{sexis} since the specificity of the functional allows us to obtain the existence of an optimal domain by applying some classical results of \cite{bdm93} for decreasing shape functionals. In Section \ref{sradial} we consider the case $f=constant$ or more generally $f(x)=\tilde f(|x|)$ with $\tilde f(r)$ decreasing; we show (see Theorem \ref{trad}) that in this situation, if $D$ is large enough, the solution of the worst-case shape optimization problem
$$\min\Big\{\E(\O)\ :\ \O\subset D,\ |\O|\le m\Big\},$$
is actually a ball of measure $m$.

The last Section \ref{snumer} contains some numerical computations on a particular example.

\section{Capacity, quasi-open sets and capacitary measures}
Here below we summarize the main tools that we use in the sequel; the interested reader can find a more detailed presentation of them in \cite{bubu05}.

\bigskip\noindent{\bf Capacity and Sobolev functions.} We define the capacity of a set $E\subset\R^d$ as
$$\cp(E)=\inf\left\{\int_{\R^d}|\nabla u|^2\,dx\ :\ u\in H^1(\R^d);\ u=1\ \text{ in a neighbourhood of }E\right\}.$$
A classical result gives that the Sobolev functions are defined up to a set of zero capacity. In fact, we have that for every $u\in H^1(\R^d)$ the set of Lebesgue points
$$\LL(u)=\left\{x_0\in\R^d\ :\ \lim_{r\to 0}\frac1{|B_r|}\int_{B_r(x_0)}u(x)\,dx\ \hbox{ exists}\right\},$$
is such that $\cp\big(\R^d\setminus\LL(u)\big)=0$. Thus, we can identify a Sobolev function with its equivalence class with respect to the relation $u\sim v$, iff $\cp(\{u\ne v\})=0$.

\bigskip\noindent{\bf Quasi-open sets and Sobolev spaces.} We say that the set $\O\subset\R^d$ is quasi-open if for every $\eps>0$ there is an open set $\O_\eps$ such that $\cp(\O_\eps)\le\eps$ and $\O\cup\O_\eps$ is open. Given a quasi-open set $\O\subset\R^d$ we define the Sobolev space $H^1_0(\O)$ as
$$H^1_0(\O)=\Big\{u\in H^1(\R^d)\ :\ \cp\big(\{u\ne0\}\setminus\O\big)=0\Big\},$$
and we notice that this definition coincides with the usual one in the case when $\O$ is open. For every quasi-open set $\O$, the Sobolev space $H^1_0(\O)$ is a closed subspace of $H^1(\R^d)$ with respect to the Sobolev norm
$$\|u\|_{H^1}=\Big(\|u\|_{L^2}^2+\|\nabla u\|_{L^2}^2\Big)^{1/2}.$$

\bigskip\noindent{\bf PDEs on quasi-open sets.} For a quasi-open set $\O\subset\R^d$ and a function $f\in L^2(\O)$ we say that $u$ is a solution of the PDE
\be\label{pdeufo}
-\Delta u=f\quad\text{in }\O,\qquad u\in H^1_0(\O),
\ee
if we have that $u\in H^1_0(\O)$ and
$$\int_\O\nabla\phi\cdot\nabla u\,dx=\int_\O\phi f\,dx,\quad\text{for every }\phi\in H^1_0(\O).$$
It is well-known that $u\in H^1_0(\O)$ is a solution of \eqref{pdeufo} if and only if it minimizes in $H^1_0(\O)$ the functional
\be\label{dirfunct}
J(\O,u,f)=\int_\O\Big(\frac12|\nabla u|^2-fu\Big)\,dx.
\ee
In this framework the maximum principle states that if $f>0$, then $\{u>0\}=\O$ up to a set of zero capacity. Thus, we can identify any quasi-open set $\O$ of finite measure with the level set $\{w_\O>0\}$ where $w_\O$ is the solution of
$$-\Delta w_\O=1\quad\text{in }\O,\qquad w_\O\in H^1_0(\O).$$
In particular, we can endow the family of all quasi-open subsets of $D$ with the metric
\be\label{metric}
d_\gamma(\O_1,\O_2)=\|w_{\O_1}-w_{\O_2}\|_{L^1(D)}.
\ee

\bigskip\noindent{\bf Capacitary measures.} We say that a positive measure $\mu$ is of a capacitary type if $\mu(E)=0$ for every set $E\subset\R^d$ such that $\cp(E)=0$. Since every $u\in H^1(\R^d)$ is defined up to a set of zero capacity, we have that the integral $\int_{\R^d}u^2\,d\mu$ is well-defined (finite or infinite). We define the Sobolev space $H^1_\mu$ as
$$H^1_\mu=\left\{u\in H^1(\R^d)\ :\ \int_{\R^d}u^2\,d\mu<+\infty\right\}.$$
Notice that if $\O$ is a quasi-open set, then  $H^1_0(\O)=H^1_{\mu_\O}$, where the capacitary measure $\mu_\O$ is defined as
$$\mu_\O(E)=\begin{cases}
0&\text{if }\cp(E\setminus\O)=0\\
+\infty&\text{if }\cp(E\setminus\O)>0.
\end{cases}$$

\bigskip\noindent{\bf PDEs involving capacitary measures.} For a capacitary measure $\mu$ and a function $f\in L^2(\R^d)$ we say that $u$ is a solution to the PDE
\be\label{mupde}
-\Delta u+\mu u=f\quad\text{in}\quad\R^d,\qquad u\in H^1_\mu,
\ee
if we have that $u\in H^1_\mu$ and
$$\int_{\R^d}\nabla\phi\cdot\nabla u\,dx+\int_{\R^d}u\phi\,d\mu=\int_{\R^d}\phi f\,dx,\quad\text{for every}\quad \phi\in H^1_\mu.$$
As in the classical case we have that $u$ is a solution of \eqref{mupde} if and only if $u$ minimizes in $H^1_\mu$ the functional
\be\label{jmu}
J(\mu,u,f)=\frac12\int_{\R^d}|\nabla u|^2\,dx+\frac12\int_{\R^d}u^2\,d\mu-\int_{\R^d}uf\,dx.
\ee
Moreover we have the following maximum principle: If $\mu_1\ge \mu_2$ (which means that $\mu_1(E)\ge\mu_2(E)$ for every quasi-open set $E$) and $f_1\le f_2$, then $u_1\le u_2$, where $u_i$ is the solution of
$$-\Delta u_i+\mu_i u_i=f_i\quad\text{in }\R^d,\qquad u_i\in H^1_{\mu_i}.$$

\bigskip\noindent{\bf The metric space of capacitary measures.} Suppose that $\Dr\subset \R^d$ is an open set of finite measure. We denote by $QO(\Dr)$ the metric space of quasi-open subsets of $\Dr$ endowed with the metric \eqref{metric}. Then the completion of $QO(\Dr)$ with respect to $d_\gamma$ is given by the set
$$\M_{\cp}(\Dr)=\Big\{\mu\text{ capacitary measure}\ :\ \mu(E)=+\infty,\text{ for every $E$ such that }\cp(E\setminus\Dr)>0\Big\},$$
with the metric
$$d_\gamma(\mu_1,\mu_2)=\|w_{\mu_1}-w_{\mu_2}\|_{L^1(D)},$$
where $w_\mu$ denotes the solution of
$$-\Delta w_\mu+\mu w_\mu=1\quad\text{in }\R^d,\qquad w_\mu\in H^1_\mu.$$
it was proved in \cite{dmmo87} that $\big(\M_{\cp}(\Dr),d_\gamma\big)$ is a compact metric space.

\bigskip\noindent{\bf Continuous functionals on $\M_{\cp}(\Dr)$.} The resolvent operator is continuous with respect to the metric above, in fact if $f_n\in L^2(\Dr)$ is a sequence converging in $L^2(\Dr)$ to $f_\infty\in L^2(\Dr)$ and $\mu_n$ is a sequence of capacitary measures converging to $\mu_\infty\in\M_{\cp}(\Dr)$ in the $d_\gamma$ distance then the sequence of solutions $u_n$ of the PDEs
$$-\Delta u_n+\mu_n u_n=f_n\quad\text{in }\Dr,\qquad u_n\in H^1_{\mu_n},$$
converges strongly in $L^2(\Dr)$ to the solution of
$$-\Delta u_\infty+\mu_\infty u_\infty=f_\infty\quad\text{in }\Dr,\qquad u_\infty\in H^1_{\mu_\infty}.$$
Thus, all the functionals of the form 
\be\label{fjmu}
F(\mu)=\int_\Dr j(x,u)\,dx,
\ee
where $j:\Dr\times\R\to\R$ is a given function, and $u$ is the solution of
$$-\Delta u+\mu u=f\quad\text{in }\Dr,\qquad u\in H^1_\mu,$$
are lower semi-continuous with respect to $d_\gamma$, provided $j(x,s)$ is continuous with respect to $s$ and bounded from below as
$$j(x,s)\ge-c|s|^2+a(x)\qquad\hbox{with $c>0$ and }a\in L^1(D).$$

\bigskip\noindent{\bf Existence of optimal sets.} The following result was proved in \cite{bdm93} and represents the main tool for proving existence of optimal domains.
\begin{teo}[Buttazzo-Dal Maso]\label{teobudm}
Suppose that $\Dr$ is a bounded open set and $\F$ is a functional on the family of quasi-open sets such that
\begin{itemize}
\item $\F$ is decreasing with respect to the set inclusion,
\item $\F$ is lower semi-continuous with respect to the $\gamma$-distance.  
\end{itemize}
Then, for every $m>0$ there is a solution to the problem 
$$\min\Big\{\F(\O)\ :\ |\O|\le m,\ \O\ \text{quasi-open},\ \O\subset\Dr\Big\}.$$
\end{teo}

\bigskip\noindent{\bf Dirichlet energy of quasi-open sets and capacitary measures.} 
Suppose that $\Dr\subset\R^d$ is a quasi-open set and $f\in L^2(\Dr)$. For a quasi-open set of finite measure $\O\subset\Dr$ we define the Dirichlet energy of $\O$ with respect to $f$ as
$$E(\O,f)=\min\Big\{J(\O,u,f)\ :\ u\in H^1_0(\O)\Big\},$$
where the functional $J(\O,u,f)$ is defined in \eqref{dirfunct}. The minimizer $u$ solves \eqref{pdeufo}. Thus multiplying both sides of \eqref{pdeufo} by $u$ and integrating by parts, we get
$$E(\O,f)=-\frac12\int_\Dr fu\,dx.$$
For a capacitary measure $\mu\in \mathcal M_{\cp}(\Dr)$ we have
$$E(\mu,f)=\min\Big\{J(\mu,u,f)\ :\ u\in H^1_\mu\Big\},$$
where $J(\mu,u,f)$ is the functional in \eqref{jmu}, and again by integration by parts we get
$$E(\mu,f)=-\frac12\int_\Dr fu\,dx.$$

\begin{oss}
The functional $E(\O,f)$ is decreasing with respect to the set inclusion since we have
$$H^1_0(\O_1)\subset H^1_0(\O_2)\qquad\hbox{whenever}\qquad\O_1\subset\O_2.$$
Moreover, $E(\O,f)$ is of the form \eqref{fjmu} and so it is lower semi-continuous with respect to $d_\gamma$.  Thus, by Theorem \ref{teobudm}, there is a solution to the shape optimization problem
$$\min\Big\{E(\O,f)\ :\ \O\ \text{quasi-open},\ \O\subset \Dr,\ |\O|\le m\Big\}.$$
In order to have a solution $\O_{opt}$ which is an open set it is necessary to assume some higher intergability of $f$, while regularity results for $\partial\O_{opt}$ are available only if we assume that $f$ is H\"older continuous.
\end{oss}

\section{Optimal domains for worst-case energy functionals }\label{sexis}

We consider a fixed bounded domain $D\subset\R^d$ and a function  $f\in L^p(D)$, where $p\in[2d/(d+2),+\infty]$ ($p>1$ if $d=2$).
For every quasi-open set $\O\subset D$ we define the worst-case shape functional \be\label{wcfunct}
\E_{\delta,p}(\O)=\sup\Big\{E(\O,f+g)\ :\ \|g\|_{L^p(D)}\le\delta\Big\},
\ee
where $\delta>0$ is a given number.

\begin{prop}\label{infsup}
For every quasi open set $\O\subset\R^d$ of finite measure and $f\in L^p(\O)$, for $p$ as above, we have
\be\label{wcfunct2}
\E_{\delta,p}(\O)=\min\left\{\int_D\Big(\frac12|\nabla u|^2-fu\Big)\,dx+\delta\|u\|_{L^{p'}(D)}\ :\ u\in H^1_0(\O)\right\}.
\ee
The supremum in \eqref{wcfunct} is attained for the function 
$$g=-\delta u|u|^{(2-p)/(p-1)}\|u\|^{-p'/p}_{L^{p'}(D)}\;,$$ 
where $p'$ is the conjugate exponent of $p$ and $u$ is the minimizer of the right hand side of \eqref{wcfunct2}.
\end{prop}

\begin{proof}
We first show that the supremum in \eqref{wcfunct} is attained. Let $g_n$ be a maximizing sequence in \eqref{wcfunct}; since $\|g_n\|_{L^p(D)}\le\delta$ we may assume, up to extracting a subsequence, that $g_n$ converges weakly to some $g$ in $L^p(D)$. The solutions $u_n$ of
$$-\Delta u_n=f+g_n\quad\hbox{in }\O,\qquad u_n\in H^1_0(\O)$$
then converge, weakly in $H^1_0(\O)$, to the solution $u$ of
$$-\Delta u=f+g\quad\hbox{in }\O,\qquad u\in H^1_0(\O).$$
Therefore,
$$\limsup_{n\to\infty}E(\O,f+g_n)\le\lim_{n\to\infty}J(\O,u,f+g_n)=J(\O,u,f+g)=E(\O,f+g).$$
Since $g_n$ is a maximizing sequence, we get the claim. 

The expression of $\E_{\delta,p}$ is of $\sup$-$\inf$ type and to the functional $J$ the $\sup$-$\inf$ switch can be applied (see for instance \cite{cla77}). The supremum of $J(\O,u,f+g)$ with respect to $g$ is easy to compute and is reached for
$$g=-\delta u|u|^{(2-p)/(p-1)}\|u\|^{-p'/p}_{L^{p'}(D)},$$
where $p'$ is the conjugate exponent of $p$; we have then
$$\E_{\delta,p}(\O)=\inf\left\{\int_D\Big(\frac12|\nabla u|^2-fu\Big)\,dx+\delta\|u\|_{L^{p'}(D)}\ :\ u\in H^1_0(\O)\right\}.$$
Note that, when $p=\infty$, the functional $\E_{\delta,\infty}$ reduces to
$$\E_{\delta,\infty}(\O)=\inf\left\{\int_D\Big(\frac12|\nabla u|^2-fu+\delta|u|\Big)\,dx\ :\ u\in H^1_0(\O)\right\}$$
and the analysis above still holds.
\end{proof}

\begin{teo}\label{exth}
The worst case shape optimization problem 
\be\label{wcshopt}
\min\Big\{\E_{\delta,p}(\O)\ :\ \O\subset\Dr,\ \O\ \text{quasi-open},\ |\O|\le m\Big\}.
\ee
admits a solution.
\end{teo}

\begin{proof}
We will prove that the functional $\E_{\delta,p}$ satisfies the hypotheses of Theorem \ref{teobudm}. By using the inclusion of Sobolev spaces $H^1_0(\O_1)\subset H^1_0(\O_2)$ whenever $\O_1\subset\O_2$ and the expression \eqref{wcfunct2} we obtain that the mapping $\O\mapsto\E_{\delta,p}(\O)$ is decreasing with respect to the set inclusion. Thus it is sufficient to prove that the shape functional $\E_{\delta,p}$ is lower semicontinuous with respect to the $\gamma$-convergence. Suppose that $\O_n$ is a sequence of quasi-open sets in $\Dr$ $\gamma$-converging to the quasi-open set $\O_\infty\subset\Dr$. For every open set $\O_n$ let $g_n\in L^p(\Dr)$ be the function such that 
$$\E_{\delta,p}(\O_n)=E(\O_n,f,g_n)\qquad\text{and}\qquad \|g_n\|_{L^p}=\delta.$$
Up to a subsequence, we may suppose that $g_n$ converges weakly in $L^p(\Dr)$ to a function $g\in L^p(\Dr)$ which is such that $\|g\|_{L^p}\le \delta$. By the $\gamma$-convergence of $\O_n$ we get that the solution $u_n$ of the equation 
$$-\Delta u_n=f+g_n\quad\text{in}\quad\O_n,\qquad u_n\in H^1_0(\O_n),$$
converges strongly in $H^1_0(\Dr)$ to the solution of 
$$-\Delta u=f+g\quad\text{in }\O_\infty,\qquad u\in H^1_0(\O_\infty).$$
Thus, we have that 
\begin{align*}
\liminf_{n\to\infty}\E_{\delta,p}(\O_n)&=\liminf_{n\to\infty}E(\O_n,f+g_n)=\liminf_{n\to\infty}-\frac12\int_\Dr (f+g_n)u_n\,dx\\
&=-\frac12\int_\Dr (f+g)u\,dx=E(\O,f+g)\ge\E_{\delta,p}(\O),
\end{align*}
which concludes the proof.
\end{proof}

\subsection{The radial case}\label{sradial}

In this section we consider the case of a right-hand side $f$ of radial type; more precisely, we assume that $f=f(|x|)$ and that the function $r\mapsto f(r)$ is decreasing. We then consider the worst-case shape optimization problem \eqref{wcshopt} where the cost functional $\E_{\delta,p}$ is given, according to Proposition \ref{infsup}, by
$$\E_{\delta,p}(\O)=\inf\left\{\int_D\Big(\frac12|\nabla u|^2-f(|x|)u\Big)\,dx+\delta\|u\|_{L^{p'}(D)}\ :\ u\in H^1_0(\O)\right\}.$$
We also assume that the domain $D$ is large enough to contain a ball of measure $m$. With the assumptions above we have the following result.

\begin{prop}\label{trad}
The worst case shape optimization problem \eqref{wcshopt} is solved by the ball of measure $m$ centered at the origin.
\end{prop}

\begin{proof}
The proof is obtained by symmetrization. It is known that the rearrangement of $\O$ into the ball $\O^*$ centered at the origin and $u$ into a radial decreasing function $u^*$ gives a lower Dirichlet integral
$$\int_{\O^*}|\nabla u^*|^2\,dx\le\int_\O|\nabla u|^2\,dx,$$
and the same $L^{p'}$ norm
$$\|u^*\|_{L^{p'}(\O^*)}=\|u\|_{L^{p'}(\O)},$$
while, due to the monotonicity assumption on $f$ and the Riesz inequality (see \cite{liebloss}), the linear term $\int_\O fu\,dx$ increases :
$$\int_{\O^*}f(|x|)u^*(x)\,dx\ge\int_\O f(|x|)u(x)\,dx.$$
Therefore, we obtain that the ball $\O^\ast$ satisfies 
$$|\O^*|=|\O|=m\qquad\text{and}\qquad\E_{\delta,p}(\O^*)\le\E_{\delta,p}(\O),$$
which concludes the proof.
\end{proof}

\section{Optimal domains for linear worst-case functionals}\label{scontrol}

In this section we consider a more general form of the optimization problem of the previous sections, presented in the form of an optimal control problem, where the control variable is the domain $\O\subset D$ in the bounded open set $\Dr\subset\R^d$. 

\bigskip\noindent{\bf The worst-case cost functional.} For a quasi-open set $\O\subset\Dr$ we consider the state equation
$$-\Delta u=f\quad\hbox{in }\O,\qquad u\in H^1_0(\O),$$
where $f\in L^p(\Dr)$ is a given fixed non-negative function and $p\ge2d/(d+2)$ ($p>1$, if $d=2$). The cost functional is then
$$F(\O,f)=\int_D j(x,u)\,dx,$$
where 
$$j(x,u)=-h(x)u\quad\hbox{with}\quad h\in L^{q}(\Dr)\quad \hbox{and}\quad h(x)\ge0,$$
where $q\ge2d/(d+2)$ ($q>1$, if $d=2$). Note that, when $h=f$ we are in the situation of Section~\ref{sexis}. We assume that the right-hand side $f$ may vary under an unknown perturbation, i.e. we consider the worst-case shape functional
$$\F_{\delta,p}(\O)=\sup_{\|g\|_{L^p}\le\delta}\inf\left\{-\int_\O h(x)u\,dx\ :\ -\Delta u=f+g \ \hbox{ in }\ \O,\ u\in H^1_0(\O)\right\}.$$

\begin{lemma}
Let $\O\subset\R^d$ be a quasi-open set of finite measure, $f\in L^p(\O)$ and  $h\in L^q(\O)$, where $p$ and $q$ are as above. Then
\be\label{mathcalFexpression}
\F_{\delta,p}(\O)=-\int_\O fw\,dx+\delta\|w\|_{L^{p'}(\O)},
\ee
where $w$ is the solution of
$$-\Delta w=h\quad\text{in }\O,\qquad w\in H^1_0(\O).$$
\end{lemma}

\begin{proof}
Let $g\in L^p(\O)$ be such that $\|g\|_{L^q}\le \delta$ and $u$ be the solution of
$$-\Delta u=f+g\quad\text{in }\O,\qquad u\in H^1_0(\O).$$
Then after integrating by parts we have
$$-\int_\O hu\,dx=\int_\O\nabla u\cdot\nabla w\,dx=-\int_\O w(f+g)\,dx\le -\int_\O wf\,dx+\delta \|w\|_{L^{p'}(\O)},$$
with an equality achieved for
$$g=-\delta w|w|^{(2-p)/(p-1)}\|w\|_{L^{p'}}^{-p'/p}\;,$$
which concludes the proof.
\end{proof}

\noindent{\bf The shape optimization problem.} We now consider the shape optimization problem 
\be\label{wcsopmain}
\min\Big\{\F_{\delta,p}(\O)\ :\ \O\subset\Dr,\ \O\ \text{quasi-open},\ |\O|\le m\Big\},
\ee
where $m>0$ and $\F_{\delta,p}(\O)$ is the functional from \eqref{mathcalFexpression}. We notice that $\F_{\delta,p}$ is not necessarily decreasing and so, we cannot apply Theorem \ref{teobudm} in order to obtain that an optimal domain exists.

\begin{teo}\label{excontr}
Let $\Dr\subset\R^d$ be a bounded open set, $h\in L^d(\Dr)$ and $f\in L^p(\Dr)$ be two given functions, where $p\ge2d/(d+2)$ ($p>1$, if $d=2$). Suppose that  $h\ge 0$ and $f\ge \eps>0$ on $\Dr$. Then, there is a constant $\bar\delta>0$ such that for every $0<\delta\le\bar\delta$, there exists a solution to the problem \eqref{wcsopmain}.
\end{teo}

\begin{proof}
Let $\O_n$ be a minimizing sequence for \eqref{wcsopmain}. By $w_n$ we denote the solution of 
$$-\Delta w_n=h\quad\text{in }\O_n,\qquad w_n\in H^1_0(\O_n).$$
We may assume that $\O_n$ converges in the $d_\gamma$ distance to the capacitary measure $\mu\in\M_{\cp}(\Dr)$. By the properties of the $\gamma$-convergence we have that  $w_n$ converges strongly in $L^2(\Dr)$ and weakly in $H^1_0(\Dr)$ to the function $w_\mu$, solution of 
$$-\Delta w_\mu+\mu w_\mu=h,\qquad w_\mu\in H^1_\mu,$$
and so, setting 
$$\F_{\delta,p}(\mu):=-\int_\Dr fw_\mu\,dx+\delta\|w_\mu\|_{p'},$$ 
we get 
$$\F_{\delta,p}(\mu)\le\liminf_{n\to\infty}\F_{\delta,p}(\O_n).$$
We now consider the quasi-open set $\O=\{w_\mu>0\}$. By the pointwise almost everywhere convergence of $w_n$ to $w_\mu$ we have that 
$$|\O|\le\liminf_{n\to+\infty}|\O_n|\le m.$$
Thus, in order to prove that the set $\O$ is a solution of \eqref{wcsopmain}, it is sufficient to prove that $\F_{\delta,p}(\O)\le\F_{\delta,p}(\mu)$. We notice that by the maximum principle we have $w_\mu\le w_\O$, so that
\begin{align*}
\F_{\delta,p}(\O)-\F_{\delta,p}(\mu)&=\int_\Dr-fw_\O\,dx+\delta\|w_\O\|_{p'}-\left(-\int_\Dr fw_\mu\,dx+\delta\|w_\mu\|_{p'}\right)\\
&=\int_\Dr-f(w_\O-w_\mu)\,dx+\delta\big(\|w_\O\|_{p'}-\|w_\mu\|_{p'}\big)\\
&\le\int_\Dr-f(w_\O-w_\mu)\,dx+\frac{\delta}{p'}\|w_\mu\|^{-p'/p}_{p'}\left(\int_\Dr w_\O^{p'}\,dx-\int_\Dr w_\mu^{p'}\,dx\right)\\
&\le\int_\Dr-f(w_\O-w_\mu)\,dx+\frac{\delta}{p'}\|w_\mu\|^{-p'/p}_{p'}\int_\Dr p'w_\O^{p'-1}(w_\O-w_\mu)\,dx\\
&\le\int_\Dr-f(w_\O-w_\mu)\,dx+\delta\int_\Dr\frac{C_h^{p'-1}}{\|w_\mu\|_{p'}^{p'-1}}(w_\O-w_\mu)\,dx,
\end{align*}
where in the last line we set $C_h=\|w_\Dr\|_\infty$, being $w_\Dr$ is the solution of 
$$-\Delta w_\Dr=h\quad\text{in }\Dr,\qquad w_\Dr\in H^1_0(\Dr),$$
and we used that by the maximum principle we have $w_\O\le w_\Dr$ and that $w_\Dr$ is bounded. Thus we have 
\be\label{6875786}
\F_{\delta,p}(\O)-\F_{\delta,p}(\mu)\le\int_\Dr\bigg(\frac{\delta C_h^{p'-1}}{\|w_\mu\|_{p'}^{p'-1}}-f\bigg)(w_\O-w_\mu)\,dx.
\ee
Our next step is to prove that the right-hand side of \eqref{6875786} is negative for $\delta$ small enough. Let $\Dr_m\subset\Dr$ be a fixed quasi-open set with $|\Dr_m|=m$ and take $\delta$ small enough, depending on $\Dr_m$, $h$, $f$, and $p$, to have $\F_{\delta,p}(\Dr_m)<0$; taking into account the definition of $\F_{\delta,p}$ in \eqref{mathcalFexpression} this is always possible. Thus for $\delta$ small enough, there is a constant $\bar c>0$ depending only on $p$, $h$, $f$ and $\Dr$ such that
$$-\bar c\ge\inf\Big\{\F_{\delta,p}(\O)\ :\ \O\subset\Dr,\ \O\ \text{quasi-open},\ |\O|\le m\Big\}.$$
In particular, we have $\F_{\delta,p}(\mu)\le-\bar c<0$. Then
$$\F_{\delta,p}(\mu)=-\int_\Dr fw_\mu\,dx+\delta\|w_\mu\|_{p'}\ge\big(\delta-\|f\|_p\big)\|w_\mu\|_{p'},$$
and we can estimate from below the norm of $w_\mu$ as follows:
$$\|w_\mu\|_{p'}\ge\frac{-\F_{\delta,p}(\mu)}{\|f\|_p-\delta}\ge\frac{\bar c}{\|f\|_p-\delta}.$$
Substituting in \eqref{6875786} we get
$$\F_{\delta,p}(\O)-\F_{\delta,p}(\mu)\le\int_\Dr\left(\delta\Big(\frac{C_h(\|f\|_p-\delta)}{\bar c}\Big)^{p'-1}-f\right)(w_\O-w_\mu)\,dx.$$
Since $f$ is bounded from below we obtain that, choosing $\bar\delta>0$ small enough, depending on $\Dr$, $f$, $h$, $m$ and $p$, for $\delta\le\bar\delta$ we have $\F_{\delta,p}(\O)-\F_{\delta,p}(\mu)\le0$, which proves the claim. 
\end{proof}

\section{A numerical example}\label{snumer}

This section is devoted to some numerical simulations of optimal solutions for the worst-case functional. We focus on the worst-case functional 
$\E_{\delta,p}$ given in (\ref{wcfunct2}) for the case $p=2$. In order to set a numerical algorithm for simulating the optimal shapes we work with the optimal control formulation, as in the previous section. Obtaining the Euler-Lagrange equation of the worst-case functional $\E_{\delta,2}$ it is elementary to check that the problem 
$$\min\Big\{\E_{\delta,2}(\O)\ :\ \O\subset D,\ |\O|\le m\Big\}$$
is equivalent to
$$\min\Big\{F(\O)\ :\ \O\subset D,\ |\O|\le m\Big\},$$
where 
$$F(\O)=\int_\O\Big[-f(x)u+\delta\|u\|_2\Big]\,dx,$$
and $u$ the solution of the state equation
$$-\Delta u=f-\delta\frac{u}{\|u\|_2}\quad\text{in }\O,\qquad u\in H^1_0(\O).$$
Remark that both, the cost functional and the state equation, are now nonlinear. Of course, when $\delta =0$ we recover the original shape optimization problem for compliance. For the numerical simulations of optimal solutions we work with the generalized (relaxed) formulation of the problem (see \cite{bdm91}), or more concretely with an approximation of it introduced in \cite[Remark 5.8]{bgrv14},
$$\min\Big\{ \tilde F(V)\ :\ V\in{\cal B}(D),\ \int_D e^{-\alpha V}\,dx\le m \Big\},$$
where $\tilde F$ stands for the functional
$$\tilde F(V)=\int_D \Big[-f(x)u+\delta \|u\|_2\Big]\,dx,$$
being $u$ the solution of the PDE
$$-\Delta u+V(x)u=f-\delta \frac{u}{\|u\|_2}\quad\text{in }\Dr,\qquad u\in H^1_0(D)$$
and ${\cal B}(\O)$ the class of nonnegative Borel measurable functions on $\O$. The constraint 
$$\int_D e^{-\alpha V}\,dx\le m$$
plays the role of the volume constraint in the original problem. In \cite[Remark 5.8]{bgrv14} it is shown that when $\alpha $ goes to zero this problem $\Gamma$-converges to the relaxed formulation of the problem given in \cite{bdm91}. In general, if an optimal shape $\O$ exists, then there is an optimal potential $V$ such that, a.e. $x\in \O$,
$$V(x)=\begin{cases}
0&\hbox{if }x\in\O\\
+\infty&\hbox{if }x\in D\setminus\O\;.
\end{cases}$$

Our simulations are performed in FreeFEM++ using the Method of Moving Asymptotes as the optimizing routing (available in FreeFEM++ through the NLopt library). Method of Moving Asymptotes is a gradient based method widely used for Topology and Structural Optimization problems \cite{svanberg}. The nonlinear state and adjoint equations are solved with a simple iterative algorithm in which the nonlinear term is updated with the state of the previous iteration. In our numerical examples we take $D=(0,1)\times (0,1)$ and a regular mesh of $50\times 50$ elements. For the numerical practice it is advisable to constraint $V$ to take vales on a bounded interval $[0,M]$, and we select $M=1000$ (when $V$ takes this maximal value on the state $u$ is very small and practically vanishes). The election of $\alpha$ is a delicate issue if we want to obtain optimal potentials $V$ in which we clearly identify the optimal shape $\O=\{V=0\}$, since it depends on the value $M$ and the number of mesh elements. In our examples we have checked the right $\alpha$ to be $\alpha=0.01$. The source term $f(x,y)$ is the piecewise constant function:
$$f(x,y)=\begin{cases}
1&\hbox{if }x\le0.5\\
2&\hbox{if }x>0.5\;.
\end{cases}$$
The volume fraction is $m=0.3$, and it is saturated in all the simulations we show. In Figure \ref{0} we show the results for the unperturbed case, i.e. $\delta=0$. In this case optimal compliance is $-0.0297465$, and in the picture we clearly identify the optimal shape $\{V=0\}$ located on the right side of the square, just where the source term $f$ is bigger. Out of the optimal shape, the potential is already big enough so that the optimal state practically vanishes, as we can see in the pictures. In Figure \ref{0.25} we show the results for the perturbed case with $\delta =0.25$. In this case, optimal compliance is $-0.0188295$, greater that the unperturbed optimal compliance, and the optimal shape is very similar to the unperturbed one, but a bit more rounded as intuitively expected. 

\begin{figure}[htp]
\begin{center}
\subfigure[Optimal potential]{\includegraphics[height=5cm]{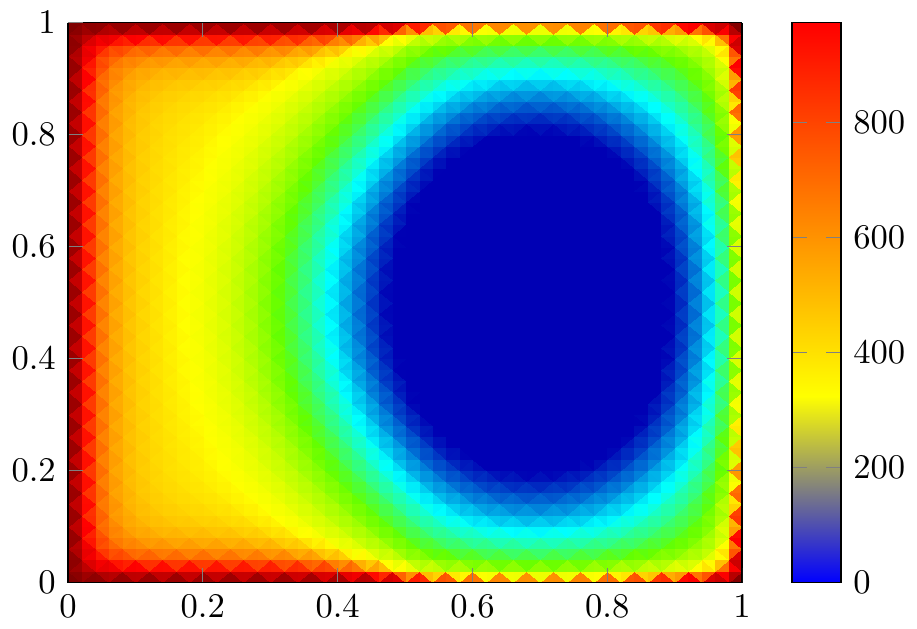}}
\subfigure[3d view optimal potential]{\includegraphics[height=5cm]{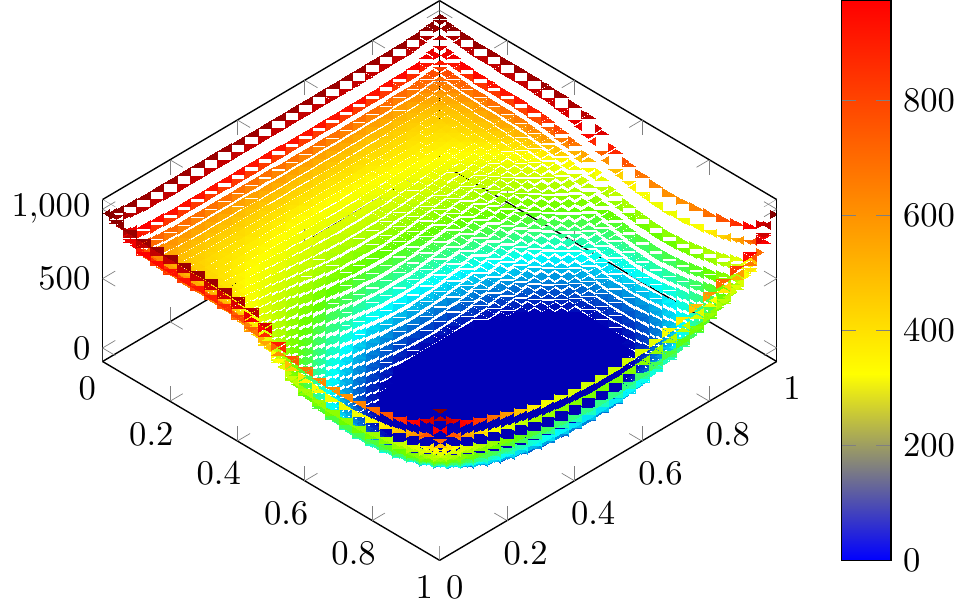}}\\
\subfigure[Optimal state]{\includegraphics[height=5cm]{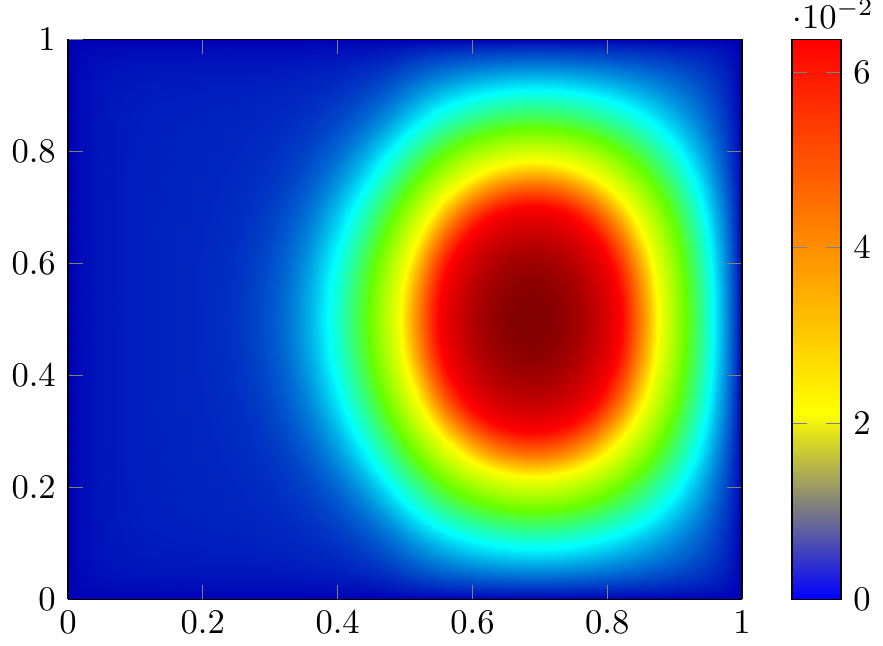}}\quad
\subfigure[3d view optimal state]{\includegraphics[height=5cm]{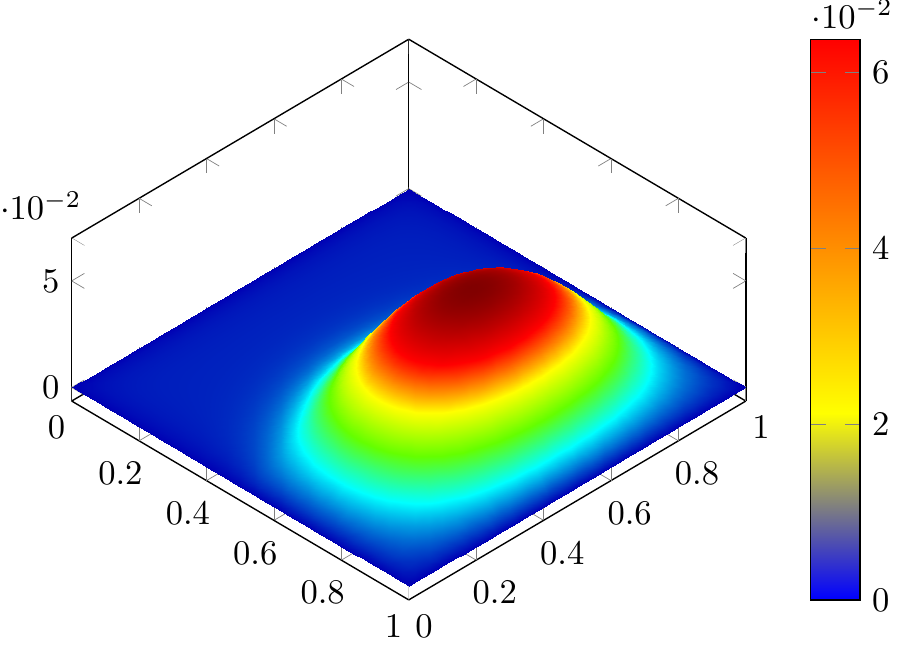}}

\caption{Results for the unperturbed case}
\label{0}
\end{center}
\end{figure}

\begin{figure}[htp]
\begin{center}
\subfigure[Optimal potential]{\includegraphics[height=5cm]{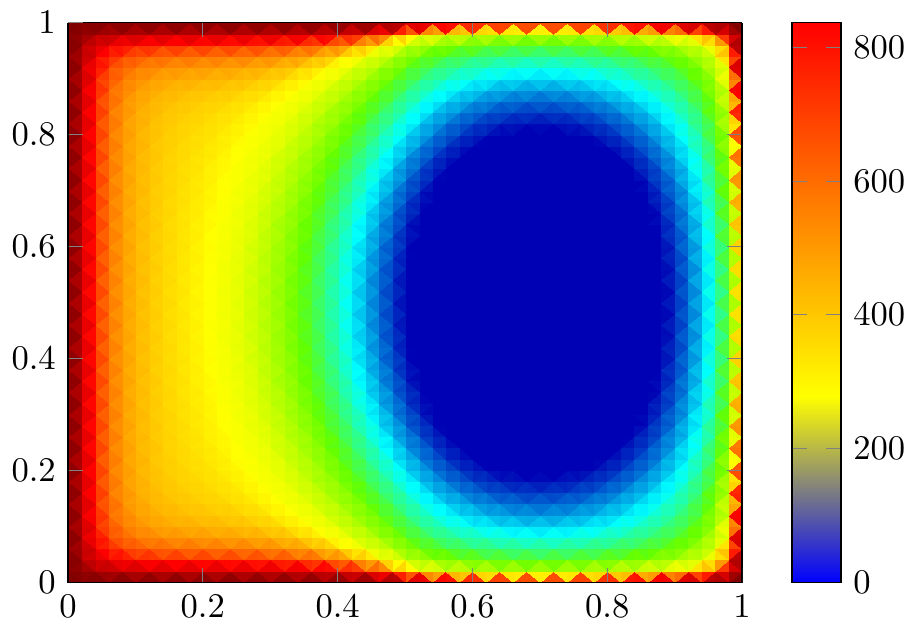}}
\subfigure[3d view optimal potential]{\includegraphics[height=5cm]{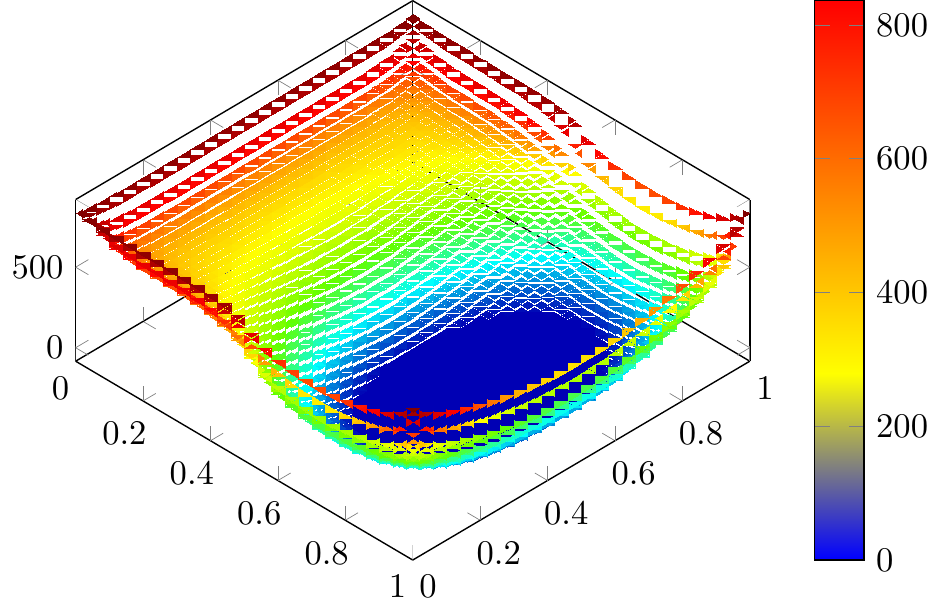}}\\
\subfigure[Optimal state]{\includegraphics[height=5cm]{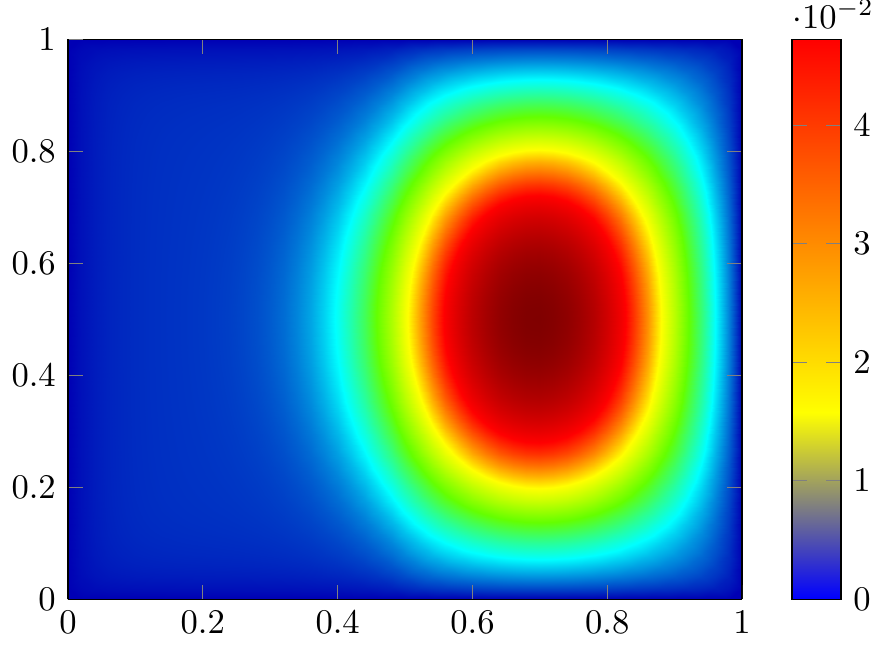}} \quad
\subfigure[3d view optimal state]{\includegraphics[height=5cm]{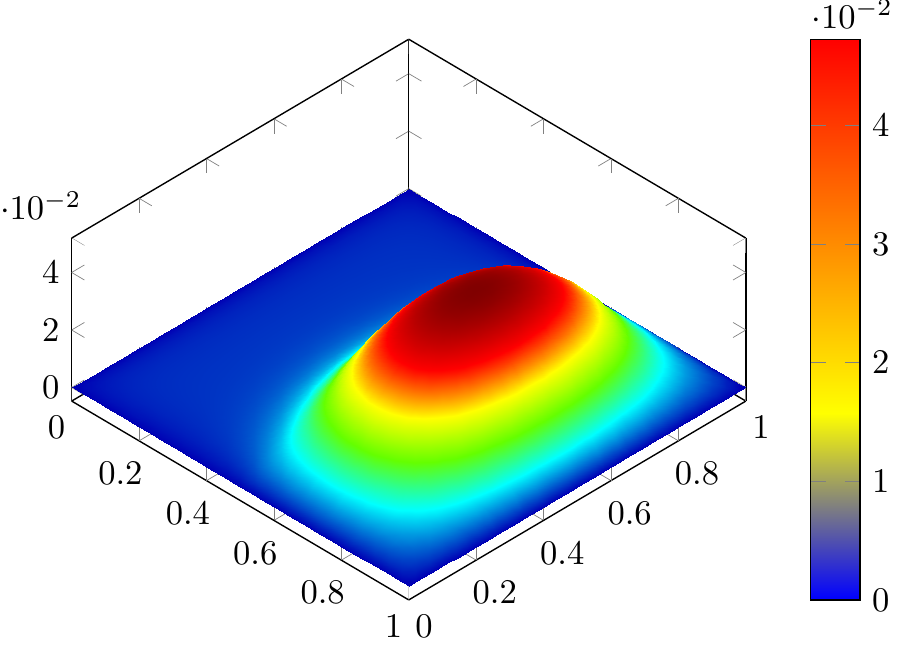}}

\caption{Results for the perturbed case with $\delta=0.25$}
\label{0.25}
\end{center}
\end{figure}

\bigskip

\ack The work of the first author is supported by the Spanish {\it Ministerio de Econom\'ia y Competitividad} through grant MTM2013-47053-P, the  {\it Junta de Castilla-La Mancha} and the European Fund for Regional Development through grant PEII-2014-010-P. J.C.B. also acknowledges the kind hospitality of the Dipartamento di Matematica of the Universit\`a di Pisa, during his stay on July 2015 and the funding support for such a stay from Universidad de Castilla-La Mancha. The work of the second author is part of the project 2010A2TFX2 {\it``Calcolo delle Variazioni''} funded by the Italian Ministry of Research and University. The second author is member of the Gruppo Nazionale per l'Analisi Matematica, la Probabilit\`a e le loro Applicazioni (GNAMPA) of the Istituto Nazionale di Alta Matematica (INdAM).


\bigskip
{\small\noindent
Jos\'e Carlos Bellido: Departamento de Matem\'aticas - ETSII, Universidad de Castilla la Mancha\\
13071 Ciudad Real - SPAIN\\
{\tt josecarlos.bellido@uclm.es}\\
{\tt http://matematicas.uclm.es/jbellido}

\bigskip\noindent
Giuseppe Buttazzo:
Dipartimento di Matematica,
Universit\`a di Pisa\\
Largo B. Pontecorvo 5,
56127 Pisa - ITALY\\
{\tt buttazzo@dm.unipi.it}\\
{\tt http://www.dm.unipi.it/pages/buttazzo/}

\bigskip\noindent
Bozhidar Velichkov:
Laboratoire Jean Kuntzmann, Universit\'e de Grenoble\\
38041 Grenoble cedex 09 - FRANCE\\
{\tt bozhidar.velichkov@gmail.com}\\
{\tt http://www.velichkov.it/}

\end{document}